\documentclass[a4paper,twoside]{article}
\usepackage{enumerate}
\usepackage{xmpmulti}
\usepackage{amssymb}
\usepackage{amsmath}
\usepackage{amsthm}
\usepackage{amscd}
\usepackage{natbib}
\usepackage{epsfig}
\usepackage{color}
\usepackage[latin5]{inputenc}
\usepackage{psfrag}
\usepackage{graphicx}

\newcommand\Z{\mathbb{Z}}

\newcommand\cA{{\mathcal A}}

\newcommand\cC{{\mathcal C}}

\newcommand\cL{{\mathcal L}}

\DeclareMathOperator{\sgn}{sgn}

\allowdisplaybreaks[1]
\newtheorem{thm}{Theorem}
\newtheorem{lem}[thm]{Lemma}

\newtheorem{corollary}[thm]{Corollary}

\theoremstyle{definition}		
\newtheorem{mydef}[thm]{Definition}
\newtheorem{example}[thm]{Example}

\newtheorem{remark}[thm]{Remark}

\newtheorem{remarks}[thm]{Remarks}

\newtheorem*{proof-sketch}{Sketch of Proof}

\graphicspath{{./}{figures/}}
\usepackage{jmsj2009}
\begin{document}

\title{Geometric intersection of curves on punctured disks}

\author{S.Öykü \textsc{Yurttaş}}



\label{startpage}

\maketitle

\abstract{We give a recipe to compute the geometric intersection number of an integral lamination with a particular type of integral lamination on an $n$-times punctured disk. This provides a way to find the geometric intersection number of two arbitrary integral laminations when combined with an algorithm of Dynnikov and Wiest.}

\section{Introduction}
Given a surface $M$ of genus $g$ with $s$ boundary components, a well known way of giving coordinates to integral laminations (i.e. a disjoint union of finitely many essential simple closed curves on $M$ modulo isotopy) and measured foliations is to use either the Dehn-Thurston coordinates or train track coordinates. See \cite{penner} for details. 

	An alternative way to coordinatize integral laminations and measured foliations on an $n$-times punctured disk $D_n$ is achieved by the \emph{Dynnikov coordinate system}. That is, Dynnikov coordinate system provides an explicit bijection between the set of integral laminations $\mathcal{L}_n$ on $D_n$ and $\mathbb{Z}^{2n-4}\setminus \{0\}$; and the set of measured foliations up to isotopy and Whitehead equivalence on $D_n$ and $\mathbb{R}^{2n-4}\setminus \{0\}$. 	
	
	Isotopy classes of orientation preserving homeomorphisms on punctured disks are described by elements of Artin's braid groups $B_n$ \cite{emil1,emil2} and the action of $B_n$ on $\mathcal{L}_n$ in terms of Dynnikov coordinates is described by the \emph{update rules} \cite{D02,M06,paper}.

	 The Dynnikov coordinate system together with the Dynnikov formulae~(update rules) was introduced in \cite{D02}. Then, it was studied in \cite{D07, DDRW02} as an efficient method for a solution of the word problem of $B_n$ and in \cite{M06,FT07,paper} for computing the topological entropy of braids. 
	 
	 In this paper, we shall use the Dynnikov coordinate system to study the geometric intersection number of two integral laminations on an $n$-times punctured disk. In particular, we shall give Theorem \ref{thm:formula} which gives a recipe to compute the geometric intersection number of an integral lamination with a particular type of integral lamination, known as a \emph{relaxed integral lamination}. This provides a way to find the geometric intersection number of two arbitrary integral laminations when combined with an algorithm of Dynnikov and Wiest \cite{wiest}, see Remark \ref{remark10}.

\section{Dynnikov Coordinates}\label{Dynnikovcoordinates}
The aim of this section is to describe the Dynnikov coordinate system for  the set of integral laminations $\cL_n$ and prove that there is an explicit bijection between $\cL_n$ and $\mathbb{Z}^{2n-4}\setminus \{0\}$.  We shall begin with the triangle coordinates which describe each integral lamination by an element of $\mathbb{Z}^{3n-5}$ using its geometric intersection number with given $3n-5$ embedded arcs in $D_n$. \emph{Dynnikov coordinates} \cite{D02} are certain linear combinations of these integers and yield a one-to-one correspondence between $\cL_n$ and $\cC_n=\mathbb{Z}^{2n-4}\setminus \{0\}$. This will be proved by Theorem \ref{lem:dynninvert} which gives the inversion of Dynnikov coordinates.

	Let $\cA_n$ be the set of arcs in~$D_n$~which have each endpoint either on the boundary or at a puncture. The arcs $\alpha_i \in \cA_n$ ($1\le i\le 2n-4$) and $\beta_i\in \cA_n$ ($1\leq i\leq n-1$) are as depicted in Figure~\ref{delta}: the arcs
$\alpha_{2i-3}$ and $\alpha_{2i-2}$ (for $2\le i\le n-1$) join the
$i^{\text{th}}$ puncture to the boundary, while the arc $\beta_i$ has
both endpoints on the boundary and passes between the $i^{\text{th}}$
and $i+1^{\text{th}}$ punctures.

\begin{figure}[h!]
\begin{center}
\psfrag{1}[tl]{$\scriptstyle{\alpha_1}$} 
\psfrag{b1}[tl]{$\scriptstyle{\beta_1}$} 
\psfrag{b2}[tl]{$\scriptstyle{\beta_{i+1}}$} 
\psfrag{bi}[tl]{$\scriptstyle{\beta_{i}}$}
\psfrag{b6}[tl]{$\scriptstyle{\beta_{n-1}}$} 
\psfrag{2i}[tl]{$\scriptstyle{\alpha_{2i+2}}$}
\psfrag{2}[tl]{$\scriptstyle{\alpha_2}$} 
\psfrag{2i-5}[tl]{$\scriptstyle{\alpha_{2i-3}}$}
\psfrag{2i-2}[tl]{$\scriptstyle{\alpha_{2i}}$}
\psfrag{2i-3}[tl]{$\scriptstyle{\alpha_{2i-1}}$}
\psfrag{2i-1}[tl]{$\scriptstyle{\alpha_{2i+1}}$}
\psfrag{2i-4}[tl]{$\scriptstyle{\alpha_{2i-2}}$}
\psfrag{n-4}[tl]{$\scriptstyle{\alpha_{2n-4}}$} 
\psfrag{n-5}[tl]{$\scriptstyle{\alpha_{2n-5}}$} 
\psfrag{d1}[tl]{$\scriptstyle{\Delta_1}$} 
\psfrag{d2}[tl]{$\scriptstyle{\Delta_2}$}
\psfrag{d3}[tl]{$\scriptstyle{\Delta_3}$}
\psfrag{dn-1}[tl]{$\scriptstyle{\Delta_{2n-4}}$} 
\psfrag{dn}[tl]{$\scriptstyle{\Delta_{2n-3}}$}
\psfrag{di}[tl]{$\scriptstyle{\Delta_{2i}}$}
\psfrag{di-1}[tl]{$\scriptstyle{\Delta_{2i-1}}$}
\psfrag{d0}[tl]{$\scriptstyle{\Delta_0}$}
\psfrag{di+1}[tl]{$\scriptstyle{\Delta_{2i+1}}$}
\psfrag{di-2}[tl]{$\scriptstyle{\Delta_{2i-2}}$}
\includegraphics[width=0.97\textwidth]{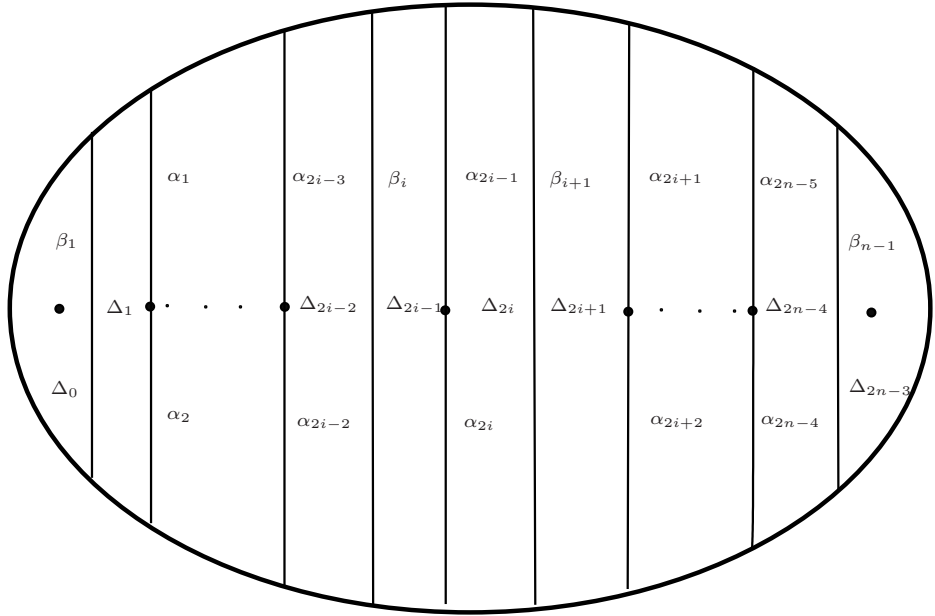}
\caption{The arcs $\alpha_i$, $\beta_i$ and triangular regions $\Delta_{i}$}\label{delta}
\end{center}
\end{figure}

Observe that the arcs divide the disk into $2n-2$ (closed) regions and $2n-4$ of these are~\emph{triangular}: Identifying the outer boundary of the disk with a point, each region on the left and right side of the $i^{\text{th}}$ puncture for $2\leq i \leq n-1$ is a triangle since it is bounded by three arcs. 

	The two triangles $\Delta_{2i-3}$ and $\Delta_{2i-2}$ on the left and right side of the $i^{\text{th}}$ puncture are defined by the arcs $\alpha_{2i-3},~\alpha_{2i-2},~\beta_{i-1}$~and ~$\alpha_{2i-3},~\alpha_{2i-2},~\beta_{i}$ respectively and the two end regions $\Delta_0$ and $\Delta_{2n-3}$ are bounded by~$\beta_1$ and $\beta_{n-1}$ respectively. See Figure~\ref{delta}.

		A naive way to describe integral laminations is achieved by triangle coordinates: Given $[\alpha]$ (the isotopy class of an arc $\alpha\in \mathcal{A}_n$ under isotopies through~$\cA_n$) and an integral lamination $\mathcal{L}$, we shall write $\alpha$ for the geometric intersection number of $\cL\in\cL_n$ with the arc~$\alpha\in\cA_n$: it will be clear from the context whether we mean the arc or the geometric intersection number assigned on the arc.
		
We also note that if $\mathcal{L}\in \mathcal{L}_n$ there is some curve system $L\in \mathcal{L}$ which is \emph{taut} (has minimum number of intersections in its homotopy class with each $\alpha_i$ and $\beta_i$). 
We fix a taut representative $L$ of a given integral lamination $\mathcal{L}\in \mathcal{L}_n$ throughout.

	For each $i$ with $1\leq i\leq n-2$, define $S_i=\Delta_{2i-1}\cup\Delta_{2i}$ (see Figure \ref{delta}). 
A \emph{path component} of $L$ in $S_i$ is a component of $L\cap S_i$. There are four types of path components in $S_i$. An \emph{above component} has end points on $\beta_i$ and $\beta_{i+1}$ and passes across $\alpha_{2i-1}$. A \emph{below component} has end points on $\beta_i$ and $\beta_{i+1}$ and passes across $\alpha_{2i}$. A \emph{left loop component} has both end points on $\beta_{i+1}$ and a \emph{right loop component} has both end points on $\beta_{i}$.

	The solid lines in Figure~\ref{pathcomponents} depict the above and below components. Left and right loop components are depicted by dashed lines.
	Note that there is one type of path component in the end regions: \emph{left loop components} in region $\Delta_0$ and \emph{right loop components} in region $\Delta_{2n-3}$.

\begin{figure}[h!]
\centering
\psfrag{2i-1}[tl]{$\scriptstyle{\alpha_{2i-1}}$} 
\psfrag{i+1}[tl]{$\scriptstyle{\beta_{i+1}}$} 
\psfrag{i}[tl]{$\scriptstyle{\beta_{i}}$} 
\psfrag{2i}[tl]{$\scriptstyle{\alpha_{2i}}$}
\psfrag{d1}[tl]{$\scriptstyle{\Delta_{2i-1}}$} 
\psfrag{d2}[tl]{$\scriptstyle{\Delta_{2i}}$} 
\psfrag{si}[tl]{$\scriptstyle{S_{i}}$} 
\includegraphics[width=0.75\textwidth]{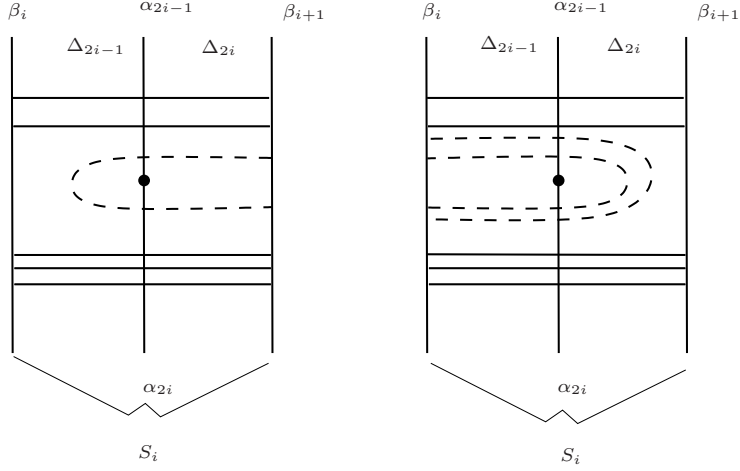}
\caption{Above, below, left loop and right loop components} \label{pathcomponents}
\end{figure}

\begin{remark}
We note that there could only be one of the two types of loop components (i.e. right or left) in each $S_i$ since the curves in $L$ are mutually disjoint. 
\end{remark}
For each $1\leq i \leq n-2$ we define 
 \begin{align}\label{bieq}
 b_i=\frac{\beta_i-\beta_{i+1}}{2}.
 \end{align}
Then $|b_i|$ gives the number of loop components in $S_i$ and $\epsilon_i=\sgn(b_i)$ tells whether the loop components are left or right. That is, when $b_i>0$ the loop components are right and when $b_i<0$ the loop components are left. See Figure~\ref{pathcomponents}: on the left, $\beta_{i+1}=\beta_i+2$ (so $b_i=-1$) and the additional two intersections of $L$ with $\beta_{i+1}$ yield one left loop component. Similarly, on the right  $\beta_{i}=\beta_{i+1}+4$ (so $b_i=2$) and the additional four intersections of $L$ with $\beta_{i}$ yield two right loop components.

The following Lemma is obvious since each above and below component intersects $\alpha_{2i-1}$ and $\alpha_{2i}$ respectively.
\begin{lem}\label{abovebelow}
The numbers of above and below components in region $S_i$ are given by $\alpha_{2i-1}-|b_i|$ and $\alpha_{2i}-|b_i|$ respectively.
\end{lem}

Similarly, the next Lemma is obvious from Figure \ref{tri22} and Figure \ref{tri21}. 

\begin{lem}\label{equalities}

There are equalities for each $S_i$: 

	When there are left loop components ($b_i< 0$), 
\begin{align}
\alpha_{2i}+\alpha_{2i-1}&=\beta_{i+1}\\
\alpha_{2i}+\alpha_{2i-1}-\beta_i&=2|b_i|,
\end{align}

when there are right loop components ($b_i>0$),
\begin{align}
\alpha_{2i}+\alpha_{2i-1}&=\beta_{i}\\
\alpha_{2i}+\alpha_{2i-1}-\beta_{i+1}&=2|b_i|,
\end{align}

and when there are no loop components ($b_i=0$),
\begin{align}
\alpha_{2i}+\alpha_{2i-1}=\beta_i=\beta_{i+1}.
\end{align}

\end{lem}
Note that Lemma \ref{equalities} implies that some coordinates are redundant.

\begin{figure}[h!]
\centering
\psfrag{a}[tl]{$\scriptstyle{\alpha_{2i-1}}$} 
\psfrag{c}[tl]{$\scriptstyle{\beta_{i+1}}$} 
\psfrag{f}[tl]{$\scriptstyle{\beta_{i}}$} 
\psfrag{b}[tl]{$\scriptstyle{\alpha_{2i}}$}
\psfrag{d1}[tl]{$\scriptstyle{\Delta_{2i-1}}$} 
\psfrag{d2}[tl]{$\scriptstyle{\Delta_{2i}}$} 
\includegraphics[width=0.75\textwidth]{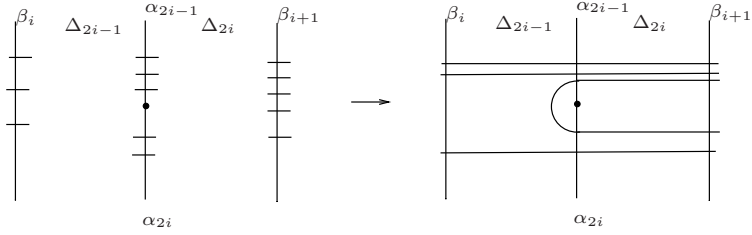}
\caption{Left loop components and the case is $b_i\leq 0 $} \label{tri22}
\end{figure}

\begin{figure}[h!]
\centering
\psfrag{a}[tl]{$\scriptstyle{\alpha_{2i-1}}$} 
\psfrag{c}[tl]{$\scriptstyle{\beta_{i+1}}$} 
\psfrag{f}[tl]{$\scriptstyle{\beta_{i}}$} 
\psfrag{b}[tl]{$\scriptstyle{\alpha_{2i}}$}
\psfrag{d1}[tl]{$\scriptstyle{\Delta_{2i-1}}$} 
\psfrag{d2}[tl]{$\scriptstyle{\Delta_{2i}}$} 
 \includegraphics[width=0.8\textwidth]{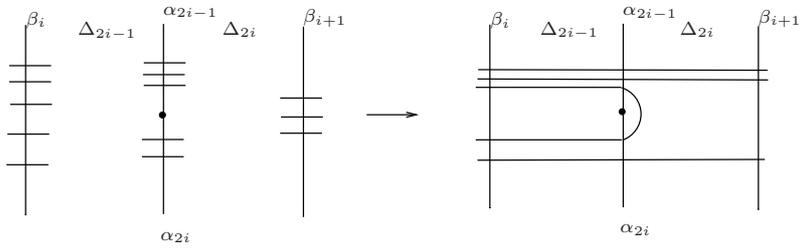}
\caption{Right loop components and the case is $b_i\geq 0$} \label{tri21}
\end{figure}

\begin{figure}[h!]
\begin{center}
\psfrag{b1}[tl]{$\scriptstyle{4}$} 
\psfrag{bi}[tl]{$\scriptstyle{8}$} 
\psfrag{b2}[tl]{$\scriptstyle{8}$}
\psfrag{b6}[tl]{$\scriptstyle{4}$} 
\psfrag{2i}[tl]{$\scriptstyle{4}$}
\psfrag{2i-5}[tl]{$\scriptstyle{2}$}
\psfrag{2i-2}[tl]{$\scriptstyle{5}$}
\psfrag{2i-3}[tl]{$\scriptstyle{3}$}
\psfrag{2i-1}[tl]{$\scriptstyle{4}$}
\psfrag{2i-4}[tl]{$\scriptstyle{6}$}
\includegraphics[width=0.7\textwidth]{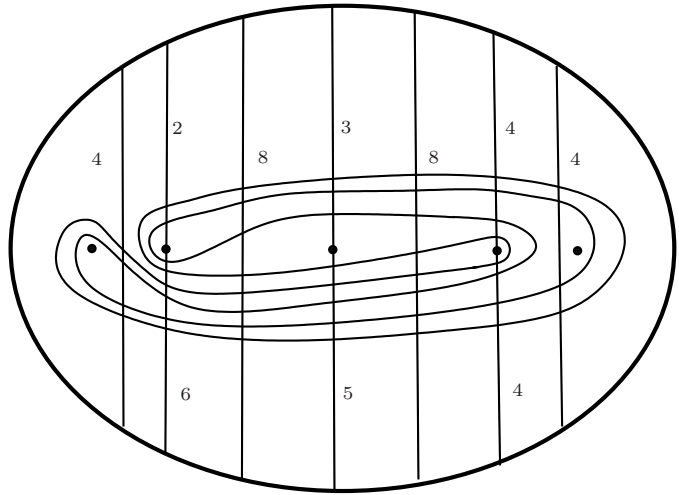}
\caption{$\tau(\mathcal{L})=(2,6,3,5,4,4;4,8,8,4)$.
} \label{munu2}
\end{center}
\end{figure}

	The triangle coordinate function $\tau:\mathcal{L}_n\to \mathbb{Z}_{\geq 0}^{3n-5}$ is defined by  
$$\tau(\mathcal{L})=(\alpha_1,\dots,\alpha_{2n-4},\beta_1,\dots,\beta_{n-1}).$$

$\tau:\mathcal{L}_n\to \mathbb{Z}_{\geq 0}^{3n-5}$ is injective: working in each region $S_i$, we can determine the number of above, below and right/left loop components. Therefore, the path components in each $S_i$ are connected in a unique way up to isotopy and hence $\mathcal{L}$ is determined uniquely.

However, it is not always possible to construct an integral lamination from given triangle coordinates. Namely, $\tau:\mathcal{L}_n\to \mathbb{Z}_{\geq 0}^{3n-5}$ is not surjective since $\tau(\mathcal{L})$ must satisfy the triangle inequality in each of the
strips of Figure~\ref{delta}, as well as additional conditions such as the equalities in Lemma \ref{equalities}. Next, we shall discuss what properties an integral lamination $\mathcal{L}\in \mathcal{L}_n$ satisfies in terms of its triangle coordinates and construct a new coordinate system from the triangle coordinates which describes integral laminations in a unique way. Namely, we shall describe the \emph{Dynnikov coordinate system} \cite{D02}.

Given a taut representative $L$ of $\mathcal{L}\in \mathcal{L}_n$ one can initially observe the following:
\begin{remarks}\label{facts}
\begin{enumerate}[i.]
\item  Every component of $L$ intersects each $\beta_i$ an even number of times. Also recall that $b_i=\frac{\beta_i-\beta_{i+1}}{2}$ and $|b_i|$ gives the number of loop  components in $S_i$. When $b_i>0$ the loop components are right and when $b_i<0$ the loop components are left (Figure \ref{looplocal}).

\item\label{three} Set $x_i=|\alpha_{2i}-\alpha_{2i-1}|$ and $m_i=\min\{\alpha_{2i-1}-|b_i|,\alpha_{2i}-|b_i|\}$; $1\leq i\leq n-2$. Then $x_i$ gives the difference between the number of above and below components in $S_i$, and $m_i$ gives the smaller of these two numbers by Lemma~\ref{abovebelow} (Figure \ref{looplocal}). We note that $x_i$ is even since each simple closed curve in $L$ intersects $\alpha_{2i}\cup\alpha_{2i-1}$ an even number of times.

\begin{figure}[h!]
\begin{center}
\psfrag{2i-1}[tl]{$\scriptstyle{\alpha_{2i-1}}$} 
\psfrag{2i}[tl]{$\scriptstyle{\alpha_{2i}}$} 
\psfrag{bi}[tl]{$\scriptstyle{\beta_i}$}
\psfrag{bi+1}[tl]{$\scriptstyle{\beta_{i+1}}$}
\psfrag{mi}[tl]{$\scriptstyle{m_i}$}
\psfrag{m}[tl]{$\scriptstyle{|b_i|}$}
\psfrag{d1}[tl]{$\scriptstyle{\Delta_{2i-1}}$}
\psfrag{d2}[tl]{$\scriptstyle{\Delta_{2i}}$}
\psfrag{xi+mi}[tl]{$\scriptstyle{x_i+m_i}$}
\includegraphics[width=0.9\textwidth]{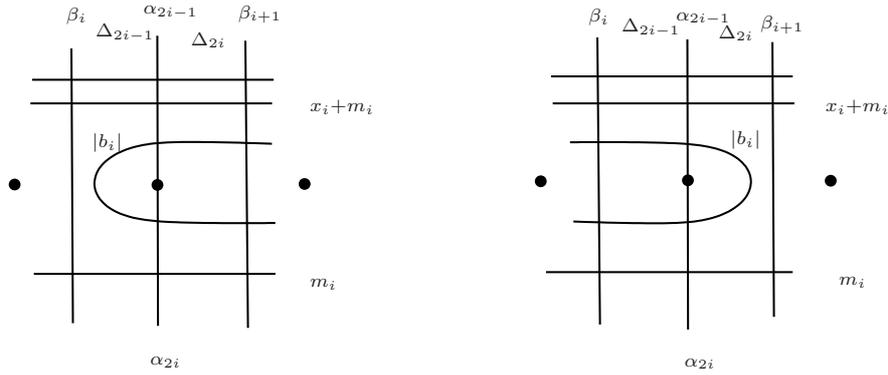}
\caption{Number of above and below components in $S_i$} 
\label{looplocal}
\end{center}
\end{figure}

\item\label{four} Set~$2a_i=\alpha_{2i}-\alpha_{2i-1}$;~$1\leq i\leq n-2$, ($a_i$ is an integer since $|a_i|=\frac{x_i}{2})$. Assume that
$b_i\geq 0$. Then, $\beta_i=\alpha_{2i}+\alpha_{2i-1}$ by Lemma \ref{equalities}. Since $2a_i=\alpha_{2i}-\alpha_{2i-1}$ it follows that

\begin{align*}
\alpha_{2i}= a_i+\frac{\beta_i}{2}; \quad \text{and}\quad \alpha_{2i-1}=-a_i+\frac{\beta_i}{2}.
\end{align*}

A similar calculation for $b_i\leq 0$ gives

\begin{align*}
\alpha_{2i}= a_i+\frac{\beta_{i+1}}{2};  \quad \text{and}\quad \alpha_{2i-1}=-a_i+\frac{\beta_{i+1}}{2}.
\end{align*}

That is to say:
\begin{eqnarray*}
\alpha_i&= \left\{ \begin{array}{ll}
         (-1)^ia _{\lceil i/2 \rceil}+\frac{\beta_{\lceil i/2\rceil}}{2}& \mbox{if $b_{\lceil i/2\rceil} \geq 0$};\\
        (-1)^i a_{\lceil i/2\rceil}+\frac{\beta_{1+\lceil i/2\rceil}}{2}& \mbox{if $b_{\lceil i/2\rceil}\leq 0$}   
         \end{array} \right. 
\end{eqnarray*} 
where $\lceil x\rceil$ denotes the smallest integer which is not less than $x$.

\item It is straightforward to compute $\beta_i$;~$1\leq i\leq n-1$ from item \ref{three}. and item \ref{four}..

\begin{align}\beta_i = \left\{ \begin{array}{ll}
         2m_i+2\left|a_i\right| & \mbox{if $b_i\leq 0$};\\
         2m_i+2\left|a_i\right|+ 2b_i & \mbox{if $b_i \geq 0$}.\end{array} \right.
         \end{align}

That is,

\begin{eqnarray*}
\beta_i=2\left[\left|a_i \right|+\max(b_i,0)+m_i\right].
\end{eqnarray*}

Since $\beta_i = \beta_1 - 2\displaystyle\sum^{i-1}_{j=1} b_j$ by (\ref{bieq}),

\begin{align*}
\beta_1=2\left[\left|a_i \right|+\max(b_i,0)+m_i+\displaystyle\sum^{i-1}_{j=1}b_j \right]\quad  \text{for}\quad 1\leq i \leq n-2.
\end{align*}
\item A crucial observation is that $m_i=0$ for some $1\leq i\leq n-1$ since otherwise there would be both above and below components in each $S_i$ and hence the integral lamination would have a curve parallel to $\partial D_n$. Then,
 
When $m_i=0$;

\begin{eqnarray*}
\beta_1=2\left[\left|a_i \right|+\max(b_i,0)+\displaystyle\sum^{i-1}_{j=1} b_j \right].
\end{eqnarray*}

When $m_i>0$;

\begin{eqnarray*}
\beta_1>2\left[\left|a_i \right|+\max(b_i,0)+\displaystyle\sum^{i-1}_{j=1} b_j\right].
\end{eqnarray*}

Therefore,

\begin{eqnarray*}
\beta_1=\max_{1\leq k \leq{n-2}}2\left[{\left|a_k \right|+\max(b_k,0)+\displaystyle\sum^{k-1}_{j=1}}b_j\right].
\end{eqnarray*}
\end{enumerate}
\end{remarks}

We have seen that $\alpha_i$ and $\beta_i$ have been recovered from $a_i$ and $b_i$ where 
$$a_i=\frac{\alpha_{2i}-\alpha_{2i-1}}{2} \quad \text{and} \quad b_i=\frac{\beta_i-\beta_{i+1}}{2}.$$

	Now, we are ready to define the \emph{Dynnikov coordinate system} which has the advantage to coordinatize $\cL_n$ bijectively and with the least number of coordinates. 

\begin{mydef}
The {\em Dynnikov coordinate function} $\rho:\cL_n\to\Z^{2n-4}\setminus\{0\}$ is defined by

\begin{eqnarray*}
\rho(\cL) = (a,b)=(a_1,\ldots,a_{n-2},\,b_1,\ldots,b_{n-2}),
\end{eqnarray*}
where for $1\le i\le n-2$

\begin{eqnarray}
a_i=\frac{\alpha_{2i}-\alpha_{2i-1}}{2}\qquad
\text{and} \qquad b_i=\frac{\beta_i-\beta_{i+1}}{2} \label{dynnikov}
\end{eqnarray}
Let $\cC_n=\Z^{2n-4}\setminus\{0\}$ denote the space of Dynnikov coordinates of integral laminations on $D_n$.
\end{mydef}

\begin{example}

The integral lamination $\mathcal{L}$ in Figure~\ref{munu2} has Dynnikov coordinates $\rho(\mathcal{L})=(2,1,0,-2,0,2).$ We have,

\begin{align*}
\alpha_1&=2,\qquad \beta_1=4,\qquad a_1=\frac{\alpha_2-\alpha_1}{2}=\frac{6-2}{2} =2\\
\alpha_2&=6,\qquad \beta_2=8,\qquad a_2=\frac{\alpha_4-\alpha_3}{2}=\frac{5-3}{2} =1\\
\alpha_3&=3,\qquad \beta_3=8,\qquad a_3=\frac{\alpha_5-\alpha_6}{2}=\frac{4-4}{2} =0\\
\alpha_4&=5,\qquad \beta_4=4,\qquad b_1=\frac{\beta_1-\beta_2}{2}= \frac{4-8}{2}=-2\\
\alpha_5&=4, \qquad \qquad \quad\qquad  ~b_2=\frac{\beta_2-\beta_3}{2}=\frac{8-8}{2} =0\\
\alpha_6&=4, \qquad \qquad \quad \qquad ~b_3=\frac{\beta_3-\beta_4}{2}=\frac{8-4}{2} =2.
\end{align*}
Note that $b_i$ can easily be read off from a picture of the lamination by counting the number of loop components and checking whether they are left or right. For example, there are two left loop components in $S_1$, therefore $b_1$ should be $-2$. 
\end{example}

\begin{thm}[Inversion of Dynnikov coordinates]
\label{lem:dynninvert}
Let $(a,b) \in \cC_n$. Then $(a,b)$ is the Dynnikov coordinate of exactly one element $\cL$ of $\cL_n$, which has 
\begin{align}
\beta_i&=2\max_{1\leq k\leq n-2}\left[|a_k|+\max(b_k,0)+\sum_{j=1}^{k-1}b_j\right]-2\sum_{j=1}^{i-1}b_j 
\label{opopp}
\end{align}
\begin{align}
\alpha_i&= \left\{ \begin{array}{ll}
         (-1)^i a_{\lceil i/2 \rceil}+\frac{\beta_{\lceil i/2\rceil}}{2}& \mbox{if $b_{\lceil i/2\rceil} \geq 0$};\\
        (-1)^i a_{\lceil i/2\rceil}+\frac{\beta_{1+\lceil i/2\rceil}}{2}& \mbox{if $b_{\lceil i/2\rceil} \leq 0$}   
         \end{array} \right.
\label{op}
\end{align}

where $\lceil x \rceil$ denotes the smallest integer which is not less than $x$. 

\end{thm}
\begin{proof}

$\rho$ is injective: Let $\mathcal{L} \in \mathcal{L}_n$, with $\tau(\mathcal{L})=(\alpha,\beta)$ and $\rho(\mathcal{L})=(a,b)$. We showed in Remarks~\ref{facts} that $(\alpha,\beta)$ must be given by (\ref{opopp}) and (\ref{op}). Hence there is no other $\mathcal{L}'\in\mathcal{L}_n$ with $\rho(\mathcal{L}')=(a,b)$ since the triangle coordinate function is injective.

$\rho$ is surjective: Let $(a,b)\in \mathcal{C}_n$. We will show that $(\alpha,\beta)$ defined by
(\ref{opopp}) and (\ref{op}) are the triangle coordinates of some $\mathcal{L} \in \mathcal{L}_n$ which
has $\rho(\mathcal{L})=(a,b)$. It is clear that if there is some $\mathcal{L}$ with $\tau(\mathcal{L})=(\alpha,\beta)$, then $\rho(\mathcal{L})=(a,b)$. By the construction in Remarks \ref{facts}, it is possible to draw in each $S_i$, $1\leq i\leq n-2$ some non-intersecting path components which intersect $\beta_i, \alpha_{2i-1}, \alpha_{2i}$, and $\beta_{i+1}$ the number of times given by $(\alpha,\beta)$. Joining these components (and completing in the only way in the two end regions) gives a system of mutually disjoint simple closed curves in $D_n$. There are no curves that bound punctures as every path component of a curve system has the property that its intersection with each $S_i$ is of one of the four types by construction, so in particular there can't be a curve that bounds a puncture.
There are no curves parallel to $\partial D_n$ as some $m_i$ is equal to zero. Hence this is an integral lamination which has triangle coordinates $(\alpha,\beta)$ as required.
\end{proof}
\textnormal{In the next section we shall give a formula to compute the geometric intersection number of a given integral lamination $\mathcal{L}\in \mathcal{L}_n$ with a given \emph{relaxed curve}~\cite{wiest} $C_{ij}$ in $D_n$ in terms of triangle coordinates. Furthermore, the formula  can be given in terms of Dynnikov coordinates by Theorem~\ref{lem:dynninvert}.}

\section{Geometric intersection of integral laminations with relaxed curves}

\begin{figure}[h!]
\begin{center}
\psfrag{s1}[tl]{$\scriptstyle{s^a_{i,j}}$} 
\psfrag{s2}[tl]{$\scriptstyle{s^b_{i,j}}$} 
\psfrag{bj+1}[tl]{$\scriptstyle{\beta_{j+1}}$} 
\psfrag{i+1}[tl]{$\scriptstyle{i+1}$} 
\psfrag{j+1}[tl]{$\scriptstyle{j+1}$} 
\psfrag{bi}[tl]{$\scriptstyle{\beta_{i}}$}
\psfrag{bj}[tl]{$\scriptstyle{\beta_{j}}$} 
\psfrag{bi+1}[tl]{$\scriptstyle{\beta_{i+1}}$} 
\psfrag{SS1}[tl]{$S_{i}$} 
\psfrag{SS2}[tl]{$S_{j}$} 
  \includegraphics[width=0.7\textwidth]{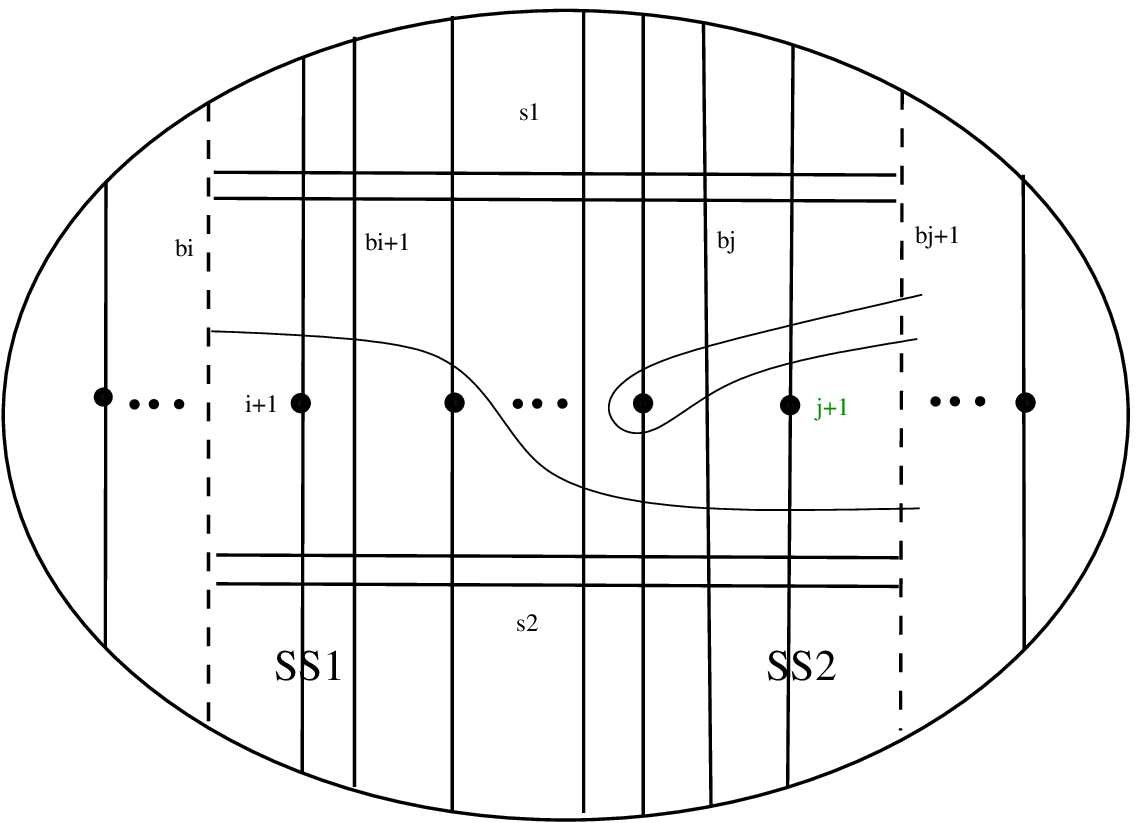}
\caption{$s^a_{i,j}$ and $s^b_{i,j}$} \label{straightpath}
\end{center}
\end{figure}

\textnormal{
Let $S_{i,j}=\displaystyle\bigcup_{i\leq k\leq j} S_k$. A \emph{path component} of $L$ in $S_{i,j}$ is a component of $L\cap S_{i,j}$. An \emph{above component} in $S_{i,j}$ has end points on $\beta_{i}$ and $\beta_{j+1}$ and does not intersect any $\alpha_{2k}$ with $i\leq k\leq j$. A \emph{below component} in $S_{i,j}$ has end points on $\beta_{i}$ and $\beta_{j+1}$ and does not intersect any $\alpha_{2k-1}$ with $i\leq k\leq j$ (Figure~\ref{straightpath}).
Using Lemma \ref{abovebelow} one can compute the number of above and below components in $S_{i,j}$.}

\begin{lem}\label{abovebelow2}
The number of above and below components in $S_{i,j}$ is given by
$$s^a_{i,j}=\displaystyle \min_{i\leq k \leq j}\{\alpha_{2k-1}-|b_k|\}\quad \text{and}\quad s^b_{i,j}=\displaystyle \min_{i\leq k \leq j}\{\alpha_{2k}-|b_k|\}$$ respectively. Therefore the sum $s_{i,j}=s^a_{i,j}+s^b_{i,j}$ gives the number of above and below components in $S_{i,j}$.
\end{lem}

\begin{proof}
For each $1\leq k \leq n-2$, $s^a_{k}=\alpha_{2k-1}-|b_k|$ and $s^b_{k}=\alpha_{2k}-|b_k|$ by Lemma~\ref{abovebelow}.

Then  $s^b_{i,j}= \displaystyle\min_{i\leq k\leq j}\{s^b_k\}$ and $s^a_{i,j}= \displaystyle\min_{i\leq k\leq j}\{s^a_k\}$.
 Hence,

$$s_{i,j}=\displaystyle\min_{i\leq k\leq j}\{s^a_k\}+\displaystyle\min_{i\leq k\leq j}\{s^b_k\}.$$

\end{proof}

\begin{remark}\label{toplampath}
\textnormal{
Notice that the number of path components in $S_{i,j}$ which are not simple closed curves is given by $\frac{\beta_i+\beta_{j+1}}{2}$ (Figure \ref{straightpath}).}
\end{remark}

\textnormal{Given an essential simple closed curve $C$ in $D_n$, $\left\|C\right\|$ denotes the minimum number of intersections of $C$ with the $x$-axis. Then, given $\mathcal{L}\in \mathcal{L}_n$, the \emph{norm }of $\mathcal{L}$ is defined as
$$\left\|\mathcal{L}\right\|=\displaystyle\sum_{i}{\left\|C_i\right\|}$$
where $\{C_i\}$ are connected components of $\mathcal{L}$.
	We say that $C_i$ is \emph{relaxed} if $\left\|C_i\right\|=2$. Then, $\mathcal{L}$ is \emph{relaxed} if each of its connected components $C_i$ is relaxed \cite{wiest}.}

\begin{figure}[h!]
\begin{center}
\includegraphics[width=0.4\textwidth]{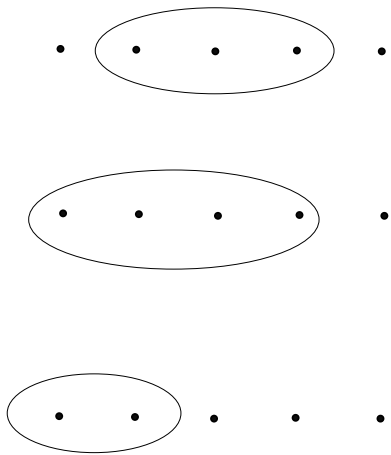}
\caption{Relaxed curves $C_{24}$,~$C_{14}$,~$C_{12}$ in $D_5$ from top to bottom} \label{relaxed}
\end{center}
\end{figure}

	\textnormal{For $1\leq i<j<n$ or $1<i<j\leq n$, $C_{ij}\in \mathcal{L}_n$ denotes the isotopy class of relaxed curves in $D_n$ which bound a disk containing the set of punctures $\{i,i+1,\dots,j\}$.}

\textnormal{Hence, we observe that $$\rho(C_{ij})=(0,\dots,0,b_1,\dots,b_{n-2})$$ where $b_{i-1}=-1$ if $i>1$, $b_{j-1}=1$ if $j<n$ and $b_k=0$ for all other cases. Figure~\ref{relaxed} shows some examples of relaxed curves.}
\begin{remark}\label{remark10}

It is always possible to turn a non-relaxed integral lamination $\mathcal{L}\in \mathcal{L}_n$ into one which is relaxed. That is to say, for any $\mathcal{L}\in \mathcal{L}_n$ there exists a braid $\beta\in B_n$ such that $\beta(\mathcal{L})$ is relaxed. An algorithm to accomplish this is given in \cite{wiest}.

\end{remark}

	\textnormal{Given $\mathcal{L}_1\in \mathcal{L}_n$ and $\mathcal{L}_2\in \mathcal{L}_n$  which are not relaxed, the geometric intersection number $i(\mathcal{L}_1,\mathcal{L}_2)$  can be computed by first relaxing one of the integral laminations with an $n$-braid $\beta$ by the algorithm described in \cite{wiest} and then computing $i(\beta(\mathcal{L}_1),\beta(\mathcal{L}_2))$ ~(note that $i(\mathcal{L}_1,\mathcal{L}_2)=i(\beta(\mathcal{L}_1),\beta(\mathcal{L}_2))$ since geometric intersection number is preserved under homeomorphisms). Hence, to compute $i(\mathcal{L}_1,\mathcal{L}_2)$, it is sufficient to find a formula that gives $i(C_{ij},\mathcal{L})$ for a given $\mathcal{L}\in \mathcal{L}_n$.}

\begin{thm}\label{thm:formula}
Given an integral lamination $\mathcal{L}\in \cL_n$ with triangle coordinates $(\alpha,\beta)$ and $C_{ij}\in \mathcal{L}_n$, $i(\mathcal{L},C_{ij})$ is given by,

\begin{align}\label{formula}
i(\mathcal{L},C_{ij})=
\beta_{i-1}+\beta_{j}-2s_{i-1,j-1}.
\end{align}
where $s_{i,j}$ is defined as in Lemma \ref{abovebelow2}.
\end{thm}

\begin{figure}[h!]
\centering
\psfrag{i+1}[tl]{$\scriptstyle{i-1}$} 
\psfrag{j+1}[tl]{$\scriptstyle{j+1}$} 
\psfrag{gammaj}[tl]{$\scriptstyle{\gamma_{ij}}$} 
\psfrag{betai-1}[tl]{$\scriptstyle{\beta_{i-1}}$} 
\psfrag{betaj}[tl]{$\scriptstyle{\beta_j}$} 
\includegraphics[width=0.9\textwidth]{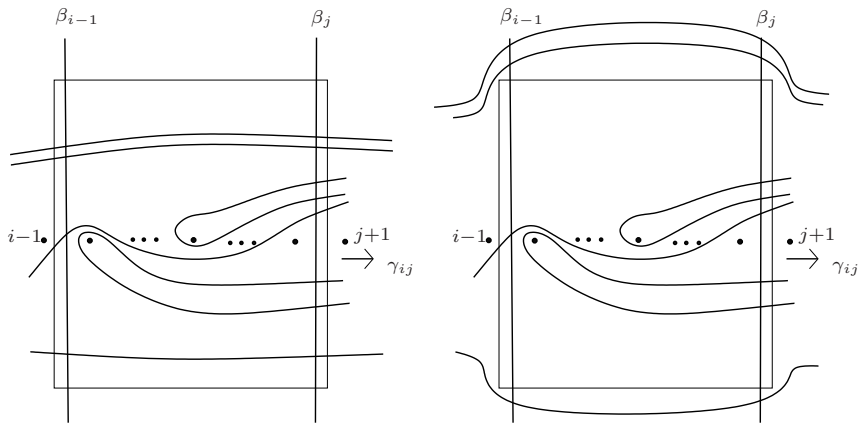}
\caption{Proof of Theorem \ref{thm:formula}}\label{straightcurve}
\end{figure}

\begin{proof}
Take a taut representative $L\in \mathcal{L}$ and a representative $\gamma_{ij}$ of $C_{ij}$ which is composed of subarcs of $\beta_{i-1}$ and $\beta_j$ and horizontal arcs which are such that the disk bounded by $\gamma_{ij}$ contains all of the path components of $L$ in $S_{i-1,j-1}$.
The number of intersections of $\gamma_{ij}$ with the path components of $L$ in $S_{i-1,j-1}$ is given by $\beta_{i-1}+\beta_j$ (See Remark \ref{toplampath}). This number can be minimized by subtracting from it the number of path components which can be isotoped so that they do not intersect $\gamma_{ij}$ any more. Such path components can only be above and below components in $S_{i-1,j-1}$ (Figure~\ref{straightcurve}). Since, each above and below component intersects $\gamma_{ij}$ twice, we have that

\begin{align*}
i(\mathcal{L},C_{ij})=\beta_{i-1}+\beta_{j}-2s_{i-1,j-1}.
\end{align*}
\end{proof}
\textnormal{
Notice that the formulae  given above can be written using Dynnikov coordinates since one can write each $\alpha_i$~and~$\beta_i$ in terms of $a_i$~and~$b_i$ by Theorem~\ref{lem:dynninvert}.} 

\begin{figure}[h!]
\begin{minipage}[b]{0.5\linewidth}
\centering
\psfrag{4}[tl]{$\scriptstyle{4}$} 
\psfrag{8}[tl]{$\scriptstyle{8}$} 
\psfrag{2}[tl]{$\scriptstyle{2}$}
\psfrag{5}[tl]{$\scriptstyle{5}$}
\psfrag{3}[tl]{$\scriptstyle{3}$}
\psfrag{6}[tl]{$\scriptstyle{6}$}
\includegraphics[width=0.9\textwidth]{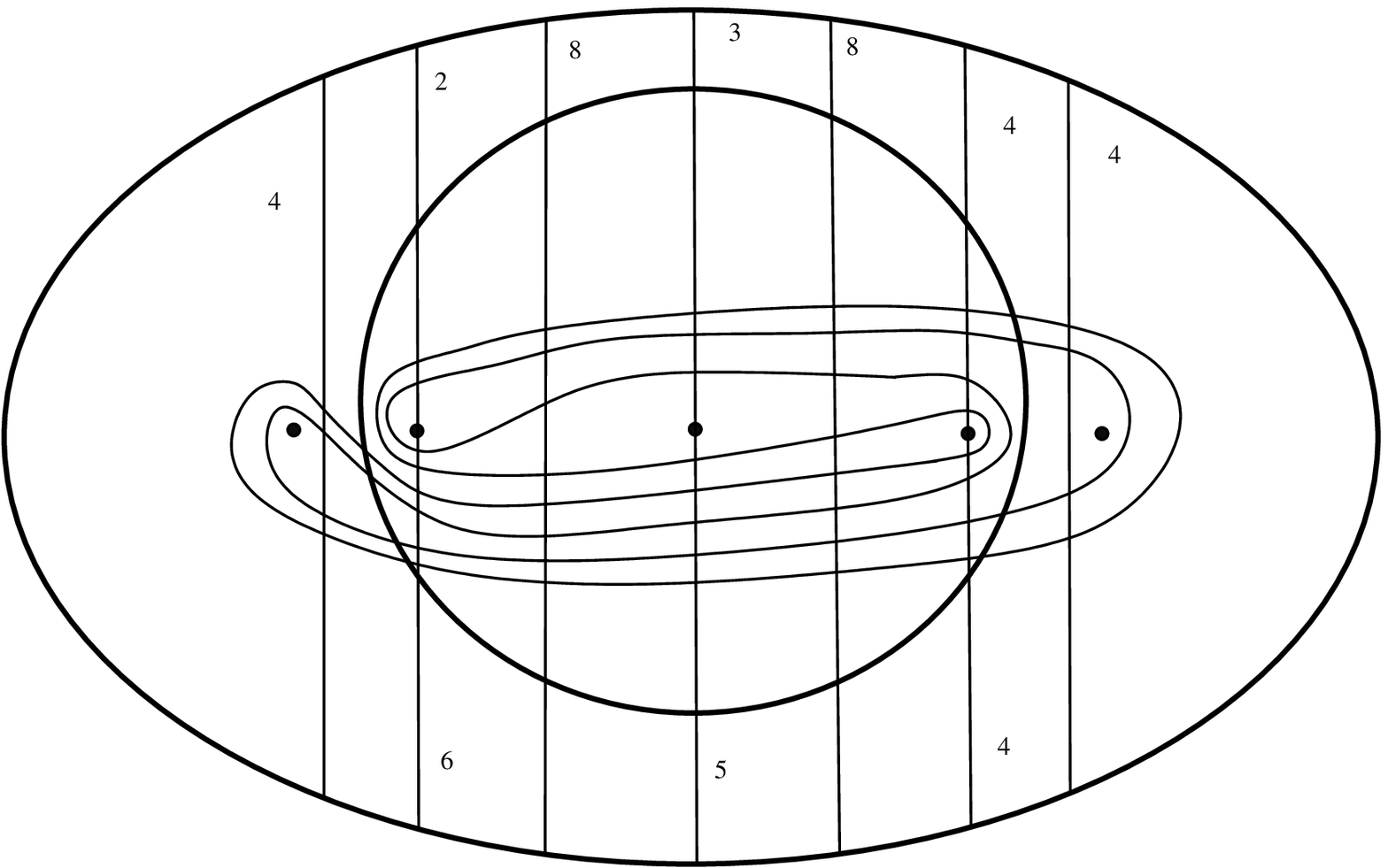}
\caption{$i(\mathcal{L},C_{ij})=4$} \label{geometricint}
\end{minipage}
\hspace{0.5cm}
\begin{minipage}[b]{0.5\linewidth}
\centering
\psfrag{4}[tl]{$\scriptstyle{4}$} 
\psfrag{8}[tl]{$\scriptstyle{8}$} 
\psfrag{2}[tl]{$\scriptstyle{2}$}
\psfrag{5}[tl]{$\scriptstyle{5}$}
\psfrag{3}[tl]{$\scriptstyle{3}$}
\psfrag{6}[tl]{$\scriptstyle{4}$}
\includegraphics[width=0.9\textwidth]{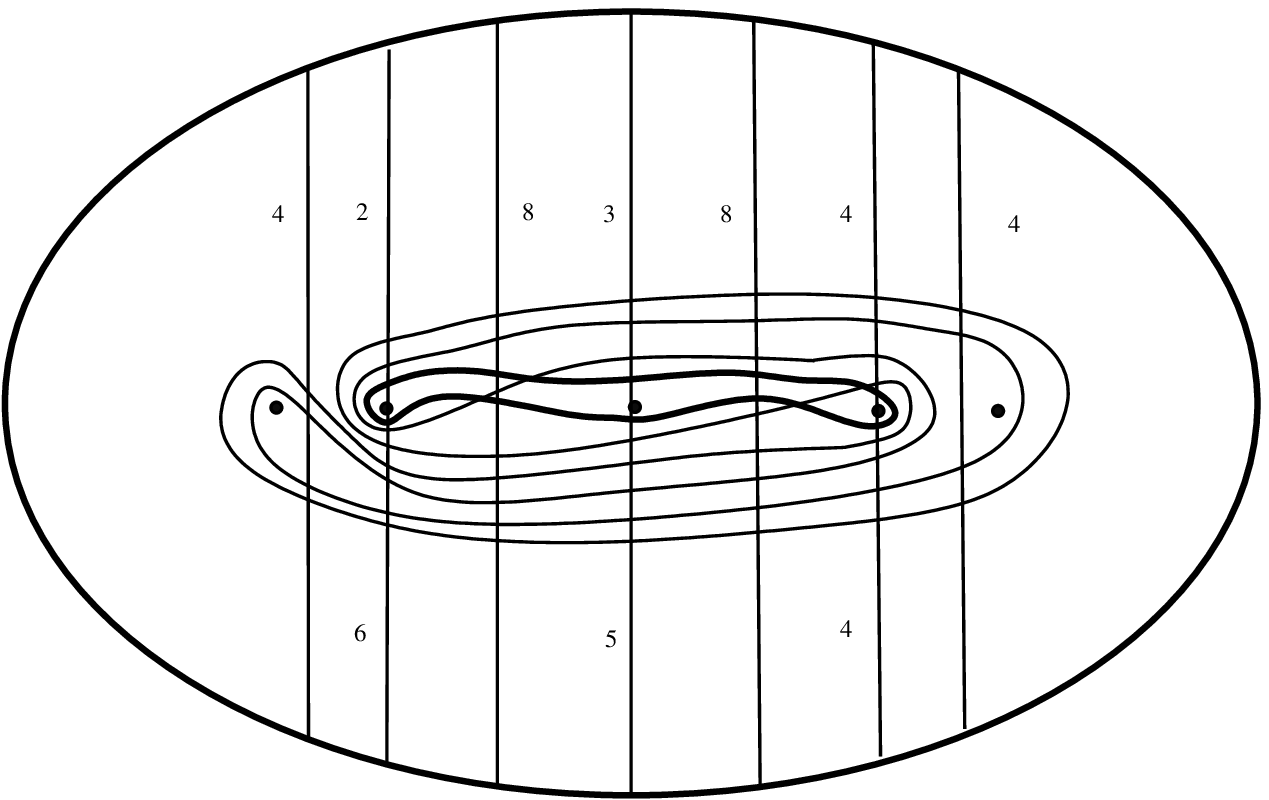}
\caption{$i(\mathcal{L},C_{ij})=4$} \label{geometricint2}
\end{minipage}
\end{figure}

\begin{example}

Let $\rho(\mathcal{L})=(2,1,0,-2,0,2)$ (Figure \ref{munu2}). We want to find $i(C_{24},\mathcal{L})$. Using the formula (\ref{formula}) we get,

\begin{align*}
i(\mathcal{L},C_{24})=\beta_{1}+\beta_{4}-2s_{1,3}.
\end{align*}
 From Theorem \ref{lem:dynninvert}, we know that $$(\alpha_1,\alpha_2,\alpha_3,\alpha_4,\alpha_5,\alpha_6;\beta_1,\beta_2,\beta_3,\beta_4)=(2,6,3,5,4,4;4,8,8,4).$$
From Lemma \ref{abovebelow2} we have, 
$$s^a_{1,3}=\displaystyle \min_{1\leq k \leq 3}\{\alpha_{2k-1}-|b_k|\}\quad \text{and}\quad s^b_{1,3}=\displaystyle \min_{1\leq k \leq 3}\{\alpha_{2k}-|b_k|\}$$
Therefore, 
\begin{align*}
s^a_{1,3}&=\displaystyle \min\{\alpha_1-|b_1|,\alpha_3-|b_2|,\alpha_5-|b_3|\}\\
&=\displaystyle \min\{2-|-2|,3-0,4-2\}=0
\intertext{and}\\
s^b_{1,3}&=\displaystyle \min\{\alpha_2-|b_1|,\alpha_4-|b_2|,\alpha_6-|b_3|\}\\
&=\displaystyle \min\{6-|-2|,5-0,4-2\}=2\\
\end{align*}
So the number of above and below components in $S_{1,3}$ equals $s^a_{1,3}+s^b_{1,3}=2$. 

Therefore,

\begin{align*}
i(\mathcal{L},C_{24})=\beta_{1}+\beta_{4}-2s_{1,3}=4+4-2\times 2=4
\end{align*}

See Figure~\ref{geometricint} and Figure~\ref{geometricint2}.
\end{example}
\begin{remark}

Observe that if $\mathcal{L}_1=\displaystyle \bigcup C_{ij}\in \cL_n$ and $\mathcal{L}_2\in \mathcal{L}_n$, then $$i(\mathcal{L}_1,\mathcal{L}_2)=\displaystyle\sum i(C_{ij},\mathcal{L}_2)$$ since the above construction can be carried out for each $C_{ij}$ in turn, working from the inside out.
\end{remark}
The next result gives the geometric intersection number of two arbitrary integral laminations on a 3-times punctured disk using Theorem \ref{thm:formula}. Again, we note that the formula  can be given in terms of Dynnikov coordinates by Theorem~\ref{lem:dynninvert}.
\begin{corollary}\label{kkk}
Let $\mathcal{L}_1\in \mathcal{L}_3$ and $\mathcal{L}_2\in \mathcal{L}_3$ have triangle coordinates $(\alpha^1,\beta^1)$ and $(\alpha^2,\beta^2)$ with Dynnikov coordinates $(a^1,b^1)$ and $(a^2,b^2)$ respectively.  Then, the geometric intersection number $i(\mathcal{L}_1,\mathcal{L}_2)$ is given by 
\begin{align}\label{xxx}
i(\mathcal{L}_1,\mathcal{L}_2)=
\begin{cases}
\alpha^1_{2}\alpha^2_{1}+\alpha^1_{1}\alpha^2_{2}~~;&\text{if}~~\epsilon^1\epsilon^2=-1\\
\left|\alpha^1_{2}\alpha^2_{1}-\alpha^1_{1}\alpha^2_{2}\right|;&\text{if}~~\epsilon^1\epsilon^2=+1
\end{cases}
\end{align}
where $\epsilon^1=\sgn(b_1^1)$ and $\epsilon^2=\sgn(b_1^2)$.
\end{corollary}

\begin{figure}[h!]
\centering
\psfrag{a}[tl]{$\scriptstyle{i(C_{12},\mathcal{L}_1)=\beta_2}=\left|\alpha_2-\alpha_1\right|$} 
\psfrag{b}[tl]{$\scriptstyle{i(C_{12},\mathcal{L}_2)=\beta_2}=\alpha_1+\alpha_2$} 
\psfrag{c}[tl]{$\scriptstyle{i(C_{23},\mathcal{L}_1)=\beta_1}=\alpha_1+\alpha_2$} 
\psfrag{d}[tl]{$\scriptstyle{i(C_{23},\mathcal{L}_2)=\beta_1}=\left|\alpha_2-\alpha_1\right|$}
\includegraphics[width=0.7\textwidth]{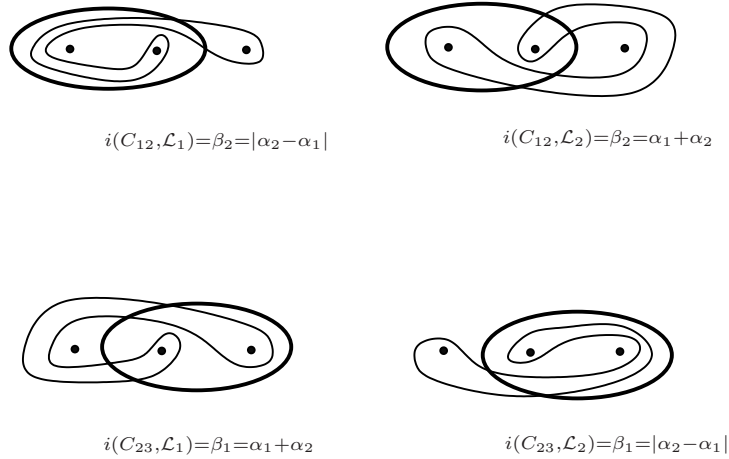}
\caption{$i(C_{ij},\mathcal{L})$ on $D_3$} \label{3puncturedcase}
\end{figure}
\begin{proof}

We first observe that the only relaxed curves in $D_3$ are $C_{12}$ and $C_{23}$. We also note that, given $\mathcal{L}\in \mathcal{L}_3$,  $i(\mathcal{L},C_{12})=\beta_2$ and $i(\mathcal{L},C_{23})=\beta_1$. Hence the formula~(\ref{xxx}) is verified for $i(\mathcal{L},C_{ij})$ by Lemma~\ref{abovebelow}. See Figure \ref{3puncturedcase}.

For the general case, we recall that $B_3$ acts on both $\mathcal{L}$ and $C_{ij}$ and there exists $\beta\in B_3$ such that $\beta(\mathcal{L})$ is either $C_{12}$ or $C_{23}$. 
Since the geometric intersection number is preserved under homeomorphisms, it follows that the formula (\ref{xxx}) is verified for $i(\mathcal{L}_1,\mathcal{L}_2)$ for any $\mathcal{L}_1\in\mathcal{L}_3$ and $\mathcal{L}_2\in \mathcal{L}_3$.
\end{proof}

\begin{example}
Let $\mathcal{L}_1$ and $\mathcal{L}_2$ be the integral laminations depicted in Figure~\ref{3puncturedcases} and; $(\alpha^1,\beta^1)$ and $(\alpha^2,\beta^2)$ be their triangle coordinates respectively.  We observe that $(\alpha^1_1,\alpha^1_2)=(3,1)$ and $(\alpha^2_1,\alpha^2_2)=(4,2)$. Since $\mathcal{L}_1$  has right loop components and $\mathcal{L}_2$ has left loop components, $\epsilon^1\epsilon^2=-1$ and hence by Corollary \ref{kkk}, $i(\mathcal{L}_1,\mathcal{L}_2)$ is given by $$i(\mathcal{L}_1,\mathcal{L}_2)=\alpha^1_{2}\alpha^2_{1}+\alpha^1_{1}\alpha^2_{2}=1\times 4+3\times 2=10.$$
\end{example}

\begin{figure}[h!]
\centering
\psfrag{l1}[tl]{$\scriptstyle{\mathcal{L}_1}$} 
\psfrag{l2}[tl]{$\scriptstyle{\mathcal{L}_2}$} 
\includegraphics[width=0.4\textwidth]{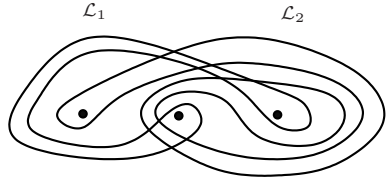}
\caption{$i(\mathcal{L}_1,\mathcal{L}_2)=10$}\label{3puncturedcases}
\end{figure}
\bigskip\noindent {\bf Acknowledgements:}
 The author would like to thank her supervisor, Dr. Toby Hall for his comments on the results of this paper most of which appeared in her Ph.D thesis.

\vspace{1cm}
\profile{S.Öykü \textsc{Yurttaş}}
{Mathematics Department,Science Faculty\\ University of Dicle, Diyarbakır, Turkey \\
saadet.yurttas@dicle.edu.tr}

\label{finishpage}

\end{document}